\documentclass[12pt,reqno]{amsart}
\usepackage{amsmath}
\usepackage{amssymb, color}
\usepackage[left=3cm,top=3cm,right=3cm,bottom=3cm]{geometry}
\usepackage{epsfig}

\begin{document}
\newtheorem{theorem}{Theorem}[section]
\newtheorem{lemma}[theorem]{Lemma}
\newtheorem{claim}[theorem]{Claim}
\newtheorem{definition}[theorem]{Definition}
\newtheorem{conjecture}[theorem]{Conjecture}
\newtheorem{proposition}[theorem]{Proposition}
\newtheorem{algorithm}[theorem]{Algorithm}
\newtheorem{corollary}[theorem]{Corollary}
\newtheorem{observation}[theorem]{Observation}
\newtheorem{problem}[theorem]{Open Problem}
\newcommand{\noin}{\noindent}
\newcommand{\ind}{\indent}
\newcommand{\al}{\alpha}
\newcommand{\om}{\omega}
\newcommand{\pp}{\mathcal P}
\newcommand{\ppp}{\mathfrak P}
\newcommand{\R}{{\mathbb R}}
\newcommand{\N}{{\mathbb N}}
\newcommand\eps{\varepsilon}
\newcommand{\E}{\mathbb E}
\newcommand{\Prob}{\mathbb{P}}
\newcommand{\pl}{\textrm{C}}
\newcommand{\dang}{\textrm{dang}}
\renewcommand{\labelenumi}{(\roman{enumi})}
\newcommand{\bc}{\bar c}
\newcommand{\G}{{\mathcal{G}}}
\newcommand{\expect}[1]{\E \left [ #1 \right ]}
\newcommand{\floor}[1]{\left \lfloor #1 \right \rfloor}
\newcommand{\ceil}[1]{\left \lceil #1 \right \rceil}
\newcommand{\of}[1]{\left( #1 \right)}
\newcommand{\set}[1]{\left\{ #1 \right\}}
\newcommand{\angs}[1]{\left\langle #1 \right\rangle}
\newcommand{\sqbs}[1]{\left[ #1 \right]}
\newcommand{\sm}{\setminus}
\newcommand{\bfrac}[2]{\of{\frac{#1}{#2}}}
\renewcommand{\k}{\kappa}
\renewcommand{\l}{\ell}
\renewcommand{\b}{\beta}
\newcommand{\blue}[1]{{\color{blue} #1}}

\title{The Total Acquisition Number of Random Graphs}

\author{Deepak Bal}
\address{Department of Mathematics, Ryerson University, Toronto, ON, Canada, M5B 2K3}
\email{deepak.c.bal@ryerson.ca}

\author{Patrick Bennett}
\address{Department of Computer Science, University of Toronto, Toronto, ON, Canada, M5S 3G4}
\email{patrickb@cs.toronto.edu}

\author{Andrzej Dudek}
\address{Department of Mathematics, Western Michigan University, Kalamazoo, MI 49008, USA}
\email{\tt andrzej.dudek@wmich.edu}
\thanks{The third author is supported in part by Simons Foundation Grant \#244712 and by a grant from the Faculty Research and Creative Activities Award (FRACAA), Western Michigan University.}

\author{Pawe\l{} Pra\l{}at}
\address{Department of Mathematics, Ryerson University, Toronto, ON, Canada, M5B 2K3}
\thanks{The fourth author is supported in part by NSERC and Ryerson University}
\email{\texttt{pralat@ryerson.ca}}

\maketitle
\begin{abstract}
Let $G$ be a graph in which each vertex initially has weight 1. In each step, the weight from a vertex $u$ to a neighbouring vertex $v$ can be moved, provided that the weight on $v$ is at least as large as the weight on $u$. The total acquisition number of $G$, denoted by $a_t(G)$, is the minimum possible size of the set of vertices with positive weight at the end of the process.

LeSaulnier, Prince, Wenger, West, and Worah asked for the minimum value of $p=p(n)$ such that $a_t(\G(n,p)) = 1$ with high probability, where $\G(n,p)$ is a binomial random graph. We show that $p = \frac{\log_2 n}{n} \approx 1.4427 \ \frac{\log n}{n}$ is a sharp threshold for this property. We also show that almost all trees $T$ satisfy $a_t(T) = \Theta(n)$, confirming a conjecture of West.
\end{abstract}

\section{Introduction}

Gossiping and broadcasting are two well studied problems involving information dissemination in a group of individuals connected by a communication network \cite{HHL}. In the gossip problem, each member has a unique piece of information which they would like to pass to everyone else. In the broadcast problem, there is a single piece  of information (starting at one member) which must be passed to every other member of the network. 
These problems have received attention from mathematicians as well as computer scientists due to their applications in distributed computing \cite{BGRV}. 
Gossip and broadcast are respectively known as ``all-to-all'' and ``one-to-all'' communication problems. In this paper, we consider the problem of acquisition, which is a type of ``all-to-one'' problem. 
Suppose each vertex of a graph begins with a weight of 1 (this can be thought of as the piece of information starting at that vertex). A \textbf{total acquisition move} is a transfer of all the weight from a vertex $v$ onto a vertex $u$, provided that immediately prior to the move, the weight on $u$ is at least the weight on $v$. Suppose a number of acquisition moves are made until no legal moves remain. Such a maximal sequence of moves is referred to as an \textbf{acquisition protocol}  and the vertices which retain positive weight after an acquisition protocol is called a \textbf{residual set}. Note that any residual set is necessarily an independent set. Given a graph $G$, we are interested in the minimum possible size of a residual set and refer to this number as the \textbf{total acquisition number of $G$}, denoted $a_t(G)$.  We are mainly concerned with the question, ``for which graphs $G$ is $a_t(G)=1$?'' {\em i.e.} when can one special member of the network acquire all the information subject to the use of total acquisition moves? The restriction to total acquisition moves can be motivated by the so-called ``smaller to larger'' rule in disjoint set data structures. For example, in the UNION-FIND data structure with linked lists, when taking a union, the smaller list should always be appended to the longer list. This heuristic improves the amortized performance over sequences of union operations.  

The parameter $a_t(G)$ was introduced by Lampert and Slater~\cite{LS} and subsequently studied in~\cite{SW, LPWWW}.  In~\cite{LS}, it is shown that $a_t(G)\le \floor{\frac{n+1}{3}}$ for any connected graph $G$ on $n$ vertices and that this bound is tight.  Slater and Wang~\cite{SW}, via a reduction to the three-dimension matching problem, show that it is NP-complete to determine whether $a_t(G)=1$ for general graphs $G$. In LeSaulnier {\em et al.}~\cite{LPWWW}, various upper bounds on the acquisition number of trees are shown in terms of the diameter and the number of vertices, $n$. They also show that $a_t(G) \le 32\log n\log\log n$ (throughout the paper, $\log n$ denotes the natural logarithm) for all graphs with diameter 2 and conjecture that the true bound is constant. For work on game variations of the parameter and variations where acquisition moves need not transfer the full weight of  vertex, see~\cite{Wen, PWW, SW2}.

\bigskip

 Randomness often plays a part in the study of information dissemination problems, usually in the form of a random network or a randomized protocol, see {\em e.g.} \cite{gos, FM, G95}. In this paper we study the total acquisition number of the \textbf{Erd\H{o}s-R\'{e}nyi-Gilbert random graph} $\G(n,p)$ where potential edges among $n$ vertices are added independently with probability $p$. We also consider the total acquisition number of random trees. Our main theorem is the following.

\begin{theorem}\label{mainthm}
Fix any $\eps > 0$. If $p=p(n) \ge \frac{1+ \eps}{\log 2} \cdot \frac{\log n}{n}$, then with high probability, $a_t(\G(n,p)) = 1$. 
\end{theorem}

In particular, by taking $p=1/2$, our result implies that while the question ``Is $a_t(G)=1$?'' is NP-complete, the answer is ``yes'' for almost all graphs. 

\bigskip

In~\cite{LPWWW}, the authors mention that understanding the behaviour of $a_t(\G(n,p))$ near the connectivity threshold, $p =\frac{\log n}{n}$, would be of particular interest. In the theory of random graphs it is usually the case that some obvious necessary condition is also a sufficient one (for example, the threshold for connectivity coincides with the one for the minimum degree at least 1; the threshold for hamiltonicity is the same as the one for the minimum degree at least 2; etc.). Hence, one could expect that $a_t(\G(n,p))=1$ already at the time a random graph becomes connected. However, it turns out that connectivity is the wrong ``obvious'' condition. Consider the following observation.


\begin{observation}\label{obs:min_degree}
If vertex $v$ is to acquire weight $w$ (at any time during the process of moving weight around), then $v$ has degree at least $\log_2 w$.
\end{observation}

\begin{proof}
Note that $v$ can only ever acquire $1 + 2 + \ldots + 2^{d(v)-1}$, in addition to the $1$ it starts with, so that is a total of $2^{d(v)}$.
\end{proof}

So if $a_t(G)=1$ then the vertex which eventually acquires all the weight must have degree at least $\log_2 n$.  Now it is true that when $p = \frac{\log n}{n}$, there exist vertices of this degree (see \cite{Bol}). But just one such vertex does not suffice; a path of significant length consisting of high degree vertices is necessary. Such a path does not exist until the expected degree exceeds $\log_2 n$. So if $p < \log_2 n / n$, then $a_t(\G(n,p)) > 1$.  In fact we prove the following stronger theorem.

\begin{theorem}\label{lowerbdthm}
 Suppose that $p = \frac {c+o(1)}{\log 2} \cdot \frac{\log n}{n}$ for some fixed  $c \in (0, 1)$. If $0 < \eps < \min\{c, 1-c\},$ then with high probability, $n^{1-c - \eps } \le a_t (\G(n,p)) \le n^{1-c + \eps }$.
\end{theorem}

This result implies that at the connectivity threshold ($p = \frac{\log n}{n}$) the total acquisition number is already of polynomial size, namely it is at least, say, $n^{0.3}$. Theorems \ref{mainthm} and \ref{lowerbdthm} together imply that $p = \frac{\log_2 n}{n}$ is the sharp threshold for the property $a_t(G)=1$.

Moreover, we prove the following theorem, confirming a conjecture of West~\cite{West1,West2}. Before we state the result, we need a few more definitions. For $n \in \N$, let $\mathcal{T}_n$ be the family of labelled trees on $n$ vertices. We say that some given property $P$ holds for \textbf{almost all trees} if the ratio between the number of trees in $\mathcal{T}_n$ with property $P$ and the total number of trees in $\mathcal{T}_n$ tends to 1 as $n \to \infty$.

\begin{theorem}\label{thm:almostalltrees}
For almost all trees $T\in \mathcal{T}_n$,
\[
a_t(T) \ge \frac{n}{3e^3}.
\]
\end{theorem}

\subsection{Notation and Conventions}
 The \textbf{random graph} $\G(n,p)$ consists of the probability space $(\Omega, \mathcal{F}, \Prob)$, where $\Omega$ is the set of all graphs with vertex set $\{1,2,\dots,n\}$, $\mathcal{F}$ is the family of all subsets of $\Omega$, and for every $G \in \Omega$,
$$
\Prob(G) = p^{|E(G)|} (1-p)^{{n \choose 2} - |E(G)|} \,.
$$
This space may be viewed as the set of outcomes of ${n \choose 2}$ independent coin flips, one for each pair $(u,v)$ of vertices, where the probability of success (that is, adding edge $uv$) is $p.$ Note that $p=p(n)$ may (and usually does) tend to zero as $n$ tends to infinity. All asymptotics throughout are as $n \rightarrow \infty $ (we emphasize that the notations $o(\cdot)$ and $O(\cdot)$ refer to functions of $n$, not necessarily positive, whose growth is bounded). We say that an event in a probability space holds \textbf{with high probability} (or \textbf{w.h.p.}) if the probability that it holds tends to $1$ as $n$ goes to infinity. We often write $\G(n,p)$ when we mean a graph drawn from the distribution $\G(n,p)$.  

All logarithms, unless otherwise noted, are assumed to be natural, {\em i.e.}\ with base $e=2.71828...$. For a vertex $v$ in a graph, we write $d(v)$ for the degree of $v$.

\bigskip

We will use the following Chernoff bound:

\begin{theorem}[\textbf{Chernoff Bound}] 
If $X$ is a binomial random variable with expectation $\mu$, and $0<\delta<1$, then $$\Pr[X < (1-\delta)\mu] \le \exp \left( -\frac{\delta^2 \mu}{2} \right)$$ and if $\delta > 0$,
\[\Pr\sqbs{X > (1+\delta)\mu} \le \exp\of{-\frac{\delta^2 \mu}{2+\delta}}.\]
\end{theorem}


\bigskip

In Section~\ref{sec:warmingup}, we prepare the reader for the proof of the main result, Theorem~\ref{mainthm}, that can be found in Section~\ref{sec:main}. Theorem~\ref{lowerbdthm} is proved in Section~\ref{sec:lowerbdthm} and Theorem~\ref{thm:almostalltrees} in Section~\ref{sec:almostalltrees}. We conclude the paper with some open problems that can be found in Section~\ref{sec:conclusion}.

\section{Warming up before attacking Theorem~\ref{mainthm}}\label{sec:warmingup}


First note that in order to prove Theorem~\ref{mainthm} it is enough to do it for  $p = \frac {1+\eps}{\log 2} \cdot \frac {\log n}{n}$ for an arbitrarily small $\eps > 0$. This follows from that fact that $a_t(G) = 1$ is an increasing graph property (see for example Lemma 1.10 in \cite{JLR}).

\bigskip

In this section, in order to prepare for a delicate and technical argument, we show that the result holds for $p = 2 \sqrt{ \log n / n}$. Let $i$ be the largest integer such that 
$$
2^i \le j = \left\lceil \sqrt{n/2} \right\rceil. 
$$
We construct a tree $T$ rooted at vertex $v$ in the following way. Vertex $v$ has $1+i+j$ children $v_0, v_1, \ldots, v_{i+j}$. Vertex $v_0$ is a leaf, and for every $1 \le \ell \le i$ we have that vertex $v_{\ell}$ has $2^{\ell}-1$ children; all remaining children of $v$ have $j$ children. Since the number of children of vertices $v_1, v_2, \ldots v_{i}$ is at most $2^{i+1} = O(\sqrt{n})$, the number of vertices in $T$ satisfies
\[\frac n2 \le j^2 \le |V(T)| \le j^2 + O(\sqrt{n}) =(1+o(1)) \frac n2. \]
It is straightforward to see that vertices of $T$ can move their weight to the root $v$. Indeed, all grandchildren of $v$ can move their weight to the corresponding parents, and then vertices $v_0, v_1, \ldots, v_i$ can send (one by one) the weight to the root (vertex $v_{\ell}$ sends the weight of $2^{\ell}$, $\ell = 0, 1, \ldots, i$). At that point of the process, $v$ has the weight of $2^{i+1} > j$ so the remaining neighbours can move their weight to $v$. 

Let $v$ be any vertex in $\G(n,p)$. First, we will show that w.h.p.\ there exists a tree $T$ rooted at $v$ that can be embedded in $\G(n,p)$. It follows from Chernoff Bound that w.h.p.\ the degree of $v$ is $(2+o(1)) \sqrt{n \log n}$. We select (arbitrarily) $1+i+j$ neighbours of $v$ and label them as $v_0, v_1, \ldots, v_{i+j}$. We continue discovering neighbours of $v_{\ell}$'s (one by one) but not every neighbour of $v_{\ell}$ will be used for the tree $T$ (there will be more neighbours than children in the corresponding subtree). Since the total number of vertices in $T$ is $(1+o(1))n/2$, there will always be at least $n/3$ vertices left that are not embedded yet.  Hence, the number of neighbours of $v_{\ell}$ that are not embedded yet is a random variable that can be lower bounded by the binomial random variable ${\rm Bin}(n/3,p)$ with expected value of $p (n/3) = (2/3) \sqrt{n \log n}$. Hence, using Chernoff Bound, with probability $1-o(n^{-1})$ there will be enough neighbours of $v_{\ell}$ to continue the process. It follows from the union bound that $T$ can be embedded w.h.p.

As we already mentioned, vertices of $T$ can move their weight to the root $v$. It is enough to show that the remaining vertices can do that too. Let $S$ be the set of neighbours of $v$ that are outside of $T$. Since 
$$
s=|S|= \deg(v) - (1+i+j) = (2+o(1)) \sqrt{n \log n} - O(\sqrt{n}) = (2+o(1)) \sqrt{n \log n},  
$$
there should be enough vertices in $S$ to dominate the rest of the graph and push the remaining weight to $v$. The important observation is that at this point of the process $v$ has weight at least $n/2$ so we do not have to control how much weight we push from a vertex of $S$ to $v$. It remains to show that $S$ dominates the remaining vertices ({\em i.e.} each remaining vertex is adjacent to a vertex in $S$) w.h.p. But this is straightforward to see, since for a given vertex we have that the probability it is not dominated is equal to
$$
(1-p)^s = \exp \left( - (1+o(1))p s \right) = \exp \left( - (4+o(1)) \log n \right) = o(n^{-1}),
$$
and the claim holds by the union bound.

\bigskip

In order to generalize these ideas to sparse graphs we have to deal with a number of problems, each of which is relatively easy to overcome but addressing all of them requires more careful argument. The spirit of the proof however, remains the same: we will define a special rooted tree (recursively) which has the property that all the weight can be moved to the root and we will show that this tree can be embedded in $\G(n,p)$ w.h.p.  This tree will not quite be spanning, but it will contain a set of vertices, $B$, which dominates the remaining vertices, $R$, and can shift their weight to the root as well. Finding a matching from $R$ to $B$ which saturates $R$ completes the proof. There have been numerous results on embedding spanning and almost spanning trees in random graphs \cite{AKS, K10, M14}, but most of these results are for embedding bounded degree trees and are not precise enough for our purposes. Let us list the most important issues and briefly describe the way we want to deal with them. Suppose that $p = \frac {1+\eps}{\log 2} \cdot \frac {\log n}{n}$.

\textbf{Problem~1:}  By Observation~\ref{obs:min_degree}, in order for a vertex to be able to accumulate a weight of $w$, it has to have degree at least $\log_2 w = \frac {1}{\log 2} \log w$. Since the average degree is only $\frac {1+ \eps}{\log 2} \log n$, it follows that (almost) every time a neighbour of the root $v$ sends its weight to $v$, the weight is (almost) doubled. In particular, some children of $v$ must be able to send a large weight to $v$, much more than the number of their children. Hence, we will need to define the tree recursively. As we will see in Corollary \ref{cor:msize}, the tree will reach level $m = (1+o(1)) \frac {\log n}{\log \log n}$.

\textbf{Problem~2:} As we already mentioned, the root and some vertices on top levels must have degrees close to the average degree in the graph. We will require that the number of children for those vertices is (roughly) $\frac{1+\epsilon/ 2}{\log 2} \log n$. However, once a positive fraction of all vertices are already embedded in the tree, the number of available ones drop substantially, so in order to be able to continue the process, we will have to decrease the required number of children to $\beta \log n$ for some $\beta$.  The bottom $\alpha \frac {\log n}{\log \log n}$ levels of the tree will have this property. The number of children on level $k$ will be denoted by $c_{m-k-1}$. (As explained below, it will be more convenient to count levels from the bottom; hence the notation $c_{m-k-1}$ instead of more natural $c_k$.)

\textbf{Problem~3:} Even though the average degree is $\frac {1+ \eps}{\log 2} \log n$, it is possible that a vertex does not have the required number of children (either $\frac{1+\epsilon/ 2}{\log 2} \log n$ or $\beta \log n$). This is not avoidable but rare, and we will show that w.h.p.\ there are at most $\sigma$ children of a given vertex that have this undesired property (we will see in the proof of Lemma \ref{lem:embed} that $\sigma = \Theta(1/\eps^2)$). Nevertheless, we have to take this into account while constructing the tree.

\bigskip

Before defining our tree, we first define a property of rooted trees which (as we will see soon) guarantees that the root can acquire all the weight on the tree, given that each vertex begins with weight 1. 

\begin{definition} [\textbf{Cut-off Property}] Let $T$ be a tree rooted at $r$. We say $T$ has the \textbf{cut-off property} if the following holds: for each vertex $v$ with children $v_1, v_2, \ldots, v_k$, and denoting by $T_i$ the subtree rooted at $v_i$, there exists an $i'$ (which may depend on $v$) so that $|T_i| = 2^{i-1}$ for $i \le i'$, and $|T_i| \le 2^{i'}$ for $i > i'$.

In this case, the vertices $v_i$ for $i \le i'$ are called \textbf{exact}. A vertex with an exact ancestor is called \textbf{tight}, and vertices that are not tight are \textbf{loose}.
\end{definition}

\begin{lemma}
If $T$ is a tree rooted at $r$ which has the cut-off property, then $a_t(T)=1$. In particular, vertex $r$ can acquire all the weight. 
\end{lemma}

\begin{proof}
We proceed by induction on the depth of $T$. The base case (depth $0$) is trivial. To see the induction step, let $r$ have children $v_1, v_2, \ldots, v_k$ and let $T_i$ be the subtree of $T$ rooted at the $v_i$. Then the $T_i$ inherit the cut-off property, and all have depth strictly less than the depth of $T$, and so by induction, all the weight from subtree $T_i$ can be loaded onto the root $v_i$. 

Now it is easy to see that by the cut-off property, $r$ may acquire the weight of each child $v_i$, going in order of increasing index. 
\end{proof}

It is time to define our recursive construction of a tree which we will have  the Cut-Off Property.
\begin{definition}\label{def:tree}
For any $\rho, m, \sigma \in \N$, and positive integer sequence $c_{m-k-1}$, construct the rooted tree $T_{\rho}$ by the following process:
\begin{enumerate}
\item Initialize: The \textbf{root} vertex $r = \angs{}$,  the \textbf{weight} $w(r) = \rho$, the \textbf{level} $k=0$. 
\item Iterate: In level $k$, if vertex $\angs{i_1, i_2, \ldots,i_k}$ has weight $w = w(\angs{i_1, i_2, \ldots i_k}) > 1$,
\begin{enumerate}
\item If $1 < w \le c_{m-k-1}$, then attach $w-1$ leaves to vertex $\angs{i_1, i_2, \ldots,i_k}$ each with weight 1.
\item  If $w > c_{m-k-1}$, then attach
  $c:=c_{m-k-1}$ children to vertex $\angs{i_1, i_2, \ldots,i_k}$, labelled $\angs{i_1, i_2, \ldots,i_k,1}, \ldots, \angs{i_1, i_2, \ldots, i_k, c}$.   Let $i'$ be the minimum integer $i \ge 0$ such that $$\displaystyle \frac{w- 2^i - \sigma }{c-i- \sigma} \le 2^i+ \sigma  .$$ Assign weights to the children as follows 
\begin{equation}\label{eq:wdistrib}
w(\angs{i_1, i_2, \ldots,i_{k+1}}) =
\begin{cases} 
2^{i_{k+1} -1} &\textrm{ if }i_{k+1} \le i' \\
1 &\textrm{ if } i' < i_{k+1} \le i'+\sigma \\
\frac{w- 2^{i'} - \sigma }{c-i'- \sigma} &otherwise.
\end{cases}
\end{equation}
Here, we assume that $\frac{w- 2^{i'} - \sigma }{c-i'- \sigma}$ is an integer and that $i' + \sigma < c$ so that $i'$ is well defined.  
\end{enumerate}
\end{enumerate}
\end{definition}

\begin{center}
\includegraphics[scale=.8]{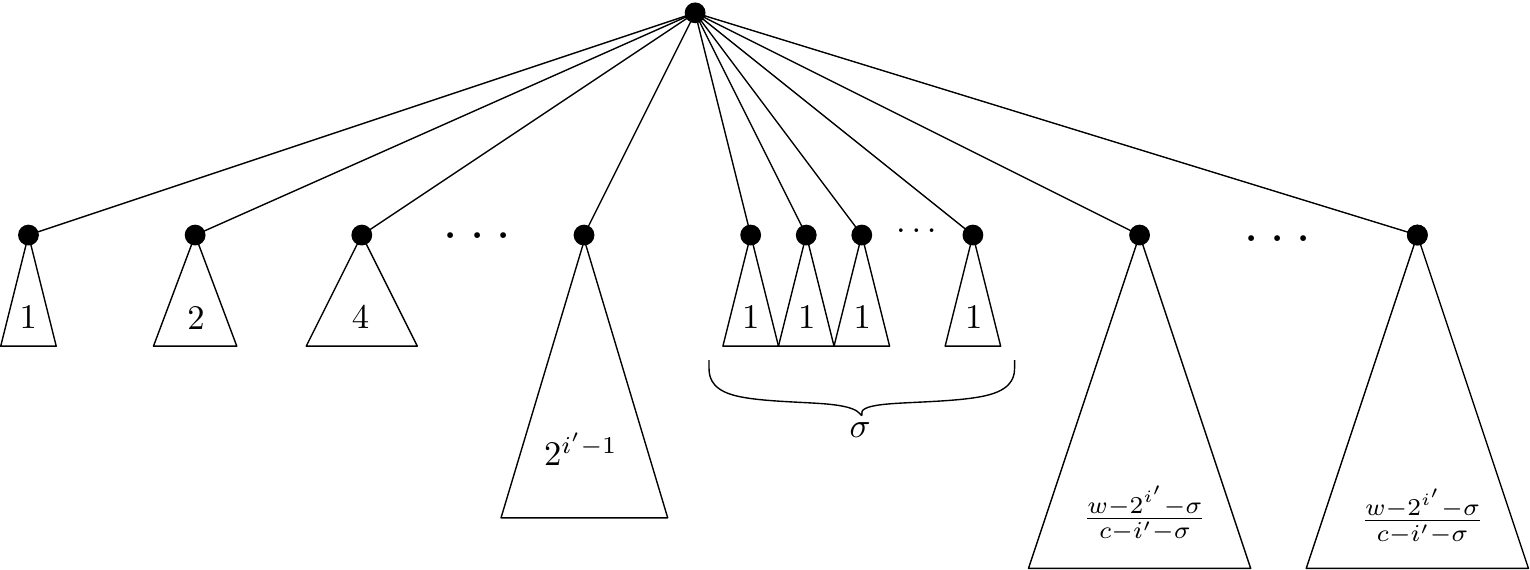}
\end{center}

Note that the tree $T_\rho$ has the Cut-Off Property. In this definition, $w(v)$ is meant to represent the number of vertices which will end up in the subtree rooted at $v$. The sequence $c$ provides a sort of threshold for the recursive part of the definition to come into play. So if the weight on $v$ is at most $c$, then the entire subtree appears in the form of leaves. If the weight on $v$ exceeds $c$, then $v$ will have exactly $c$ children and the weight is distributed according to \eqref{eq:wdistrib}. 

\bigskip

\textbf{Problem~4:} In the recursive definition of the tree, we distribute the remaining weight equally among some number of vertices.  Hence, we will have to make sure that certain divisibility conditions hold. Unfortunately, it is difficult to find the initial weight $\rho$ that does it. Hence, in order to do that, we first define $c^*_{k}$ to be the desired number of children of vertices on level $k+1$ (this time, counted from the bottom), and then assign weights starting from the bottom level and doing calculations upwards. This issue is addressed by Definition \ref{def:sequence}.

\textbf{Problem~5:} Our goal will be to construct a tree that consists of $(\frac 85+o(1))n$ vertices, that is, a tree $T_\rho$ with $\rho = (\frac 85+o(1))n$. (Of course, $\G(n,p)$ has only $n$ vertices; $T_\rho$ is an abstract tree that will be ``trimmed'' before embedding it into a random graph.) As we already mentioned, for a fixed sequence of $c^*_{k}$'s and $\sigma$, one can easily (recursively) calculate the weight $\rho_{k}$ of loose vertices on level $k$ (again, counted from the bottom), and the weight of the root $\rho = \rho_{m}$. However, it is hard to expect that the desired condition holds, namely, that $\rho_m = (\frac 85+o(1))n$. In order to solve this problem we start with any sequence $c^*_k$, take $m$ to be the largest integer such that $\rho_m \le \frac 85 n$, and then modify the sequence slightly to get the desired sequence $c_k$ with $\rho_m =  (\frac 85+o(1))n$. Let us note that a non-constructive argument is used here that shows only the existence; the sequence $c_k$ is not explicitly defined. Lemmas \ref{lem:rhoj} and \ref{lem:i'} provide useful relationships between the sequences $c$ and $\rho$ which aid in the proof of Lemma \ref{lem:adjustc} which proves the existence of the desired sequence $c$.

\textbf{Problem~6:} It is difficult to expect that a given tree on $n$ vertices can be embedded in a random graph. Hence, we are going to remove a number of leaves in $T_\rho$ to get another tree $T_\rho'$ on $(\frac 45+o(1))n$ vertices that can be embedded in $\G(n,p)$ w.h.p. The important property will be that parents of removed leaves can not only dominate the remaining $(\frac 15+o(1))n$ vertices but also can push all the weight to the root. This issue is addressed by Definition \ref{def:bereft} and Lemma \ref{lem:mostlybereft}.

\section{Proof of Theorem~\ref{mainthm}}\label{sec:main}

Set $d= \frac{1+ \eps}{\log 2}  \log n$ and let $\sigma, \b, \alpha$ be constants. 
\begin{definition}\label{def:sequence}
Let \[c_j ^* := \begin{cases} 
\b \log n  &\textrm{ if } j \le \alpha \frac{\log n}{\log \log n} \\
  \frac{1+\frac{ \eps}{2}}{\log 2} \log n&otherwise.
 \end{cases}\]
and let $c$ be a sequence such that $c_j ^* \le c_j \le c_j ^* (1+o(1))$. Define the function 
$$
i^*(x):= 
\begin{cases}
0 &\textrm{ if }x \le \sigma \\
\lceil \log_2 (x - \sigma) \rceil &\textrm{ otherwise. } 
\end{cases}
$$
Finally, define sequences $\rho_1, \rho_2, \ldots$ and $b_1, b_2, \ldots$ recursively by putting $\rho_1 :=2$ and 
$$
\rho_{j+1} := \sigma + 2^{i^*(\rho_j)}  +  \Big( c_j - i^*(\rho_j) - \sigma \Big) \cdot  \rho_j,
$$
$b_1:=1$ and  
$$
b_{j+1} :=  \Big( c_j - i^*(\rho_j) - \sigma \Big)\cdot b_j.
$$ 
\end{definition}
 
Note that the sequences $\rho, b$ depend on our choice of the sequence $c$ (we assume that constants $\eps, \sigma, \b, \alpha$ are fixed in advance). As was mentioned before (see Problem~4), the main purpose of this recursive sequence is to calculate (for a given sequence $c$ and depth $m$) the weight of the root; in fact, $\rho_j$ is the weight of each loose vertex at level $j$ (counted from the bottom) so the weight of the whole tree is $\rho_m$. Let us also mention that the purpose of $i'$ in Definition~\ref{def:tree} was to make sure that the total weight of subtrees rooted at exact vertices together with the weight of the root is at least the weight of each subtree rooted at non-exact children. It is straightforward to see that $i^*$ in Definition~\ref{def:sequence} has the same purpose and so these values are always the same. Finally, let us point out that we fix $\rho_1 = 2$ which indicates that every loose vertex at the level directly above the bottom has precisely one leaf. These leaves will play an important role in our argument and $b_j$ counts how many such leaves we have in the tree rooted at loose vertex at level $j$ (as usual, counted from the bottom).

Let $\rho^*$ be the sequence corresponding to $c^*$. Let  $m$ be the largest integer so that   $\rho^*_m \le \frac{8}{5} n$. Note that for every  $j \ge 2$, $\rho_j \ge \rho_2 = \Omega(\log n)$ and so
\begin{equation}\label{eq:nice_recurrence} 
\rho_{j+1} = \Big( c_j - \log_2 \rho_j + O(1) \Big) \cdot  \rho_j \le c_j \rho_j.
\end{equation}
It follows that $\rho^*_m = \Omega \of{\frac{n}{\log n}}$, since $\rho$ grows by at most a log factor each time. 

\bigskip
 
Henceforth we will keep $m = m(n)$ as defined above \eqref{eq:nice_recurrence} and  consider the sequences $\rho, b$ only up to the terms $\rho_m, b_m$. We will consider sequences $c$ with terms that might be larger than those of $c^*$. As a result, $\rho_j \ge \rho^*_j$ for all $j$. However, we will only consider sequences $c$ such that the corresponding sequence $\rho$ has $\rho_m = n^{1+o(1)}$.  

\begin{lemma} \label{lem:rhoj}
Let $c$ be any sequence such that $c^*_j \le c_j \le c^*_j (1+o(1))$ and the $\rho$-sequence corresponding to $c$ has $\rho_m = n^{1+o(1)}$. If  $\alpha < \frac{\b \log 2}{2}$ then 
$$
\rho_j = \exp \{ (j-1) \log \log n + O(j) \}
$$ 
for all $2 \le j \le m$. 
\end{lemma}
\begin{proof}
Clearly, $\rho_2 = \Theta(\log n)$ and it follows immediately from~(\ref{eq:nice_recurrence}) that for every  $2 \le j < m$ we have
$$
\frac{\rho_{j+1}}{\rho_j} \le c_j \le 2 c^*_j \le 2 \max \left( \b, \frac{1+\frac{ \eps}{2}}{\log 2} \right) \log n.
$$
Hence, $\rho_j \le \exp \{ (j-1) \log \log n + O(j) \}$ for all $j \le m$ and so the upper bound holds. In particular, as long as $j \le \alpha \frac{ \log n}{\log \log n}$ we have 
\begin{equation}\label{eq:rho_alpha}
\rho_j \le \rho_{\alpha \frac{ \log n}{\log \log n}} \le n^{\alpha \cdot(1+o(1))}.
\end{equation}

For the lower bound, we use~(\ref{eq:nice_recurrence}) one more time and~(\ref{eq:rho_alpha}) to observe that for every $j$ such that  $2 \le j \le \alpha \frac{ \log n}{\log \log n}$ we have 
$$
\frac{\rho_{j+1}}{\rho_j} \ge \left( \beta (1+o(1)) - \frac {\alpha}{\log 2} (1+o(1)) \right) \log n \ge \frac {\beta}{2} \log n
$$
(by our choice of $\alpha$). Since for every $j \le m$ we have $\rho_j \le \rho_m = n^{1+o(1)}$, for every $j$ such that $\alpha \frac{ \log n}{\log \log n} < j <m$ we have 
$$
\frac{\rho_{j+1}}{\rho_j} \ge \left( \frac{1+\frac{ \eps}{2}+o(1)}{\log 2} - \frac{1+o(1)}{\log 2} \right) \log n \ge \frac {\eps}{4 \log 2} \log n.
$$
It follows that $\rho_j \ge \exp \{ (j-1) \log \log n + O(j) \}$ for all $j \le m$ and the proof is finished.
\end{proof}
Henceforth, we assume that $\alpha < \frac{\beta \log 2}{2}$.
We immediately get the following corollary.

\begin{corollary} \label{cor:msize}
$m = (1+o(1)) \frac{\log n}{ \log \log n}$. 
\end{corollary}

In Definition~\ref{def:tree}, it was assumed that $i'$ was well defined, that is, that the condition $i' + \sigma < c$ holds. Because of the relationship between the two definitions, this condition is equivalent to the condition $c_j - i^*(\rho_j) - \sigma \ge 1$ in Definition~\ref{def:sequence}. In the next lemma, we show that the same condition for $\alpha$ as in previous lemmas is enough to guarantee that $i'$ is well defined.
The following is a useful property which we will use in the next few arguments.
\begin{lemma}\label{lem:i'}
We have that $i'$ always exists. In fact, the following stronger property holds: for every $j$ such that  $1 \le j \le m$ we have 
$$
c_j - \log_2 \rho_j = \Omega(\log n).
$$
\end{lemma}

\begin{proof}
In order to show that $i'$ exists, we will show that the equivalent condition that $c_j - i^*(\rho_j) - \sigma \ge 1$ in Definition~\ref{def:sequence} holds. In fact, we will show something stronger, namely, that $c_j - \log_2 \rho_j = \Omega(\log n)$ for every  $1 \le j \le m$. For $j \le \alpha \frac{ \log n}{\log \log n}$ we have 
$$
c_j - \log_2 \rho_j \ge c_j - \log_2 \rho_{\alpha \frac{ \log n}{\log \log n}} = (1+o(1)) \left( \beta - \frac {\alpha}{\log 2} \right) \log n = \Omega(\log n).
$$
If $j > \alpha \frac{ \log n}{\log \log n}$, then
$$
c_j - \log_2 \rho_j \ge c_j - \log_2 \rho_m = (1+o(1)) \left( \frac{1+\frac{ \eps}{2}}{\log 2}  - \frac {1}{\log 2} \right) \log n = \Omega(\log n).
$$
The proof is complete.
\end{proof}

Our next task is to show that one can adjust a sequence $c^*$ slightly to get another sequence $c$ with $\rho_m = \of{\frac{8}{5} + o(1) } n$.

\begin{lemma}\label{lem:adjustc}
There exists a sequence of integers $c_j$ with $c_j ^* \le c_j \le c_j ^* (1+o(1))$ and $\rho_m = \of{\frac{8}{5} + o(1) } n $.
\end{lemma}

\begin{proof}
We will start by setting $c_j = c_j ^*$ for all $j$, and then apply a number of operations to the sequence $c$. Each operation will consist of increasing a single term $c_j$ by $1$ and leaving all other terms the same. 

Suppose that the sequence $\tilde{c}$ agrees with sequence $c$ except in the $j_0$ term where we have $\tilde{c}_{j_0} = c_{j_0} + 1$. Let $\tilde{\rho}$ and $\rho$ be the corresponding sequences, which must then agree for all  $j \le j_0$.  Then, it follows from~(\ref{eq:nice_recurrence}) and Lemma \ref{lem:i'} that 
\begin{eqnarray*}
\frac{\tilde{\rho}_{j_0+1}}{\rho_{j_0+1}} &=& \frac{ \Big( (c_{j_0}+1) - \log_2 {\rho}_{j_0} + O(1) \Big) \cdot  {\rho}_{j_0}}{\Big( c_{j_0} - \log_2 \rho_{j_0} + O(1) \Big) \cdot  \rho_{j_0}} \\
&=& 1+ O\of{\frac{1}{c_{j_0} - \log_2 \rho_{j_0}} } = 1+ O\of{\frac{1}{\log n} }.
\end{eqnarray*}
Now, since $\tilde{\rho}_j \ge \rho_j$, for $j \ge j_0+1$ we have 
\begin{eqnarray*}
\frac{\tilde{\rho}_{j+1}}{\rho_{j+1}} &=& \frac{\Big( c_j - \log_2 \tilde{\rho}_j + O(1) \Big) \cdot  \tilde{\rho}_j}{\Big( c_j - \log_2 \rho_j + O(1) \Big) \cdot \rho_j} \le \frac{\Big( c_j - \log_2 \rho_j + O(1) \Big) \cdot  \tilde{\rho}_j}{\Big( c_j - \log_2 \rho_j + O(1) \Big) \cdot \rho_j}\\
&=& \frac{\tilde{\rho}_{j}}{\rho_{j}} \cdot \of{1+ O\of{\frac{1}{c_j - \log_2 \rho_j} }} = \frac{\tilde{\rho}_{j}}{\rho_{j}} \cdot \of{1+ O\of{\frac{1}{\log n} }}.
\end{eqnarray*}
Hence, 
$$
\frac{\tilde{\rho}_{m}}{\rho_{m}} \le \of{1+ O\of{\frac{1}{\log n} }}^m = 1 + O\of{\frac{1}{\log \log n} }. 
$$
In other words, each time we increment a term of sequence $c$, the effect on $\rho_m$ is negligible. 
However we will now show that if we perform this operation on $c$ enough times (while still not changing it too much each time), the effect on $\rho_m$ can be as much as we need it to be.

Suppose now that $\tilde{c}_j = c_j^* + \frac{ \log n}{\log \log \log n} = c_j^* (1+o(1))$ for all $j$. Our goal is to show that $\tilde{\rho}_m \ge \frac {8}{5} n$. For a contradiction, suppose that it is not the case, that is, $\tilde{\rho}_m < \frac {8}{5} n$. Note that we have 
$$
\log_2 \tilde{\rho}_j = \log_2 \rho^*_j + O(j) 
$$ 
since from Lemma~\ref{lem:rhoj} it follows that both $\tilde{\rho}_j$ and  $\rho^*_j$ are equal to $\exp \{ (j-1) \log \log n + O(j) \}.$ Using~(\ref{eq:nice_recurrence}) as usual, by Lemma \ref{lem:i'} we get that 
\begin{align*}
\frac{\tilde{\rho}_{j+1}}{\rho^*_{j+1}}  &= \frac{\Big( c_j +  \frac{ \log n}{\log \log \log n} - \log_2 \tilde{\rho}_j + O(1) \Big) \cdot  \tilde{\rho}_j}{\Big( c_j - \log_2 \rho^*_j + O(1) \Big) \cdot  \rho^*_j}\\
& = \frac{   c_j- \log_2 \rho^*_j + \frac{ \log n}{\log \log \log n}  +O(j) }{ c_j - \log_2 \rho^*_j +O(1) } \cdot \frac{\tilde{\rho}_j}{\rho^*_j}\\
& = \frac{\tilde{\rho}_{j}}{\rho^*_{j}} \of{ 1+ \Theta \of{\frac{\log n}{(c_j- \log_2 \rho^*_j) \log \log \log n}}}\\
& = \frac{\tilde{\rho}_{j}}{\rho^*_{j}} \of{ 1+ \Theta \of{\frac{1}{\log \log \log n}}} = \frac{\tilde{\rho}_{j}}{\rho^*_{j}} \exp \left \{ \Theta \of{\frac{1}{\log \log \log n}} \right \}
\end{align*}
And so we have 
\begin{eqnarray*}
\tilde{\rho}_m &=& \rho^*_m \cdot \exp \left \{ \Theta \of{\frac{m}{\log \log \log n}} \right \} \\
&=& \Omega \left( \frac {n}{\log n} \right) \cdot \exp \left \{ \Theta \of{\frac{\log n}{(\log \log n) (\log \log \log n)}} \right \} \gg n
\end{eqnarray*}
which is a contradiction and so the sequence $\tilde{c}$ is such that $\tilde{\rho}_m \ge \frac {8}{5} n$.  

Thus we can apply the operation ``increment one term by $1$" to the sequence $c=c^*$ several times so that each term gets increased by at most $\frac{ \log n}{\log \log \log n} = o(\log n)$, and we are able to do so in such a manner that $\rho_m = \of{\frac{8}{5} + o(1)} n$. The proof is finished.
\end{proof}

\begin{definition}\label{def:bereft}
Define the tree $T_{\rho_j} '$ to be the tree $T_{\rho_j}$ with each leaf in the bottom level being removed if it has a loose parent. Call the parents that lose their children \textbf{bereft}.
\end{definition}

Note that by induction and definition of $\rho_j, b_j$, and $T_{\rho_j} '$, we see that $T_{\rho_j} '$ has $\rho_j - b_j$ many vertices, $b_j$ of which  are bereft. It is not difficult to see that by construction, $T_{\rho_j} '$ has the cut-off property. Moreover, if we form another tree $T_{\rho_j} ''$  by re-attaching at most one leaf to each bereft parent of $T_{\rho_j} '$, then $T_{\rho_j}''$ still has the cut-off property.

\bigskip

Our next goal is to show that almost all vertices of $T_{\rho_j}'$ are bereft. Since each bereft vertex has exactly one child in $T_{\rho_j}$, we get that $|T_{\rho_j}| = (2+o(1)) |T_{\rho_j}'|$. In particular, $|T_{\rho_j}'| = (\frac 45 + o(1)) n$.

\begin{lemma}\label{lem:mostlybereft}
For  all $1\le j \le m$ we have $\frac{b_j}{\rho_j -  b_j} \rightarrow 1$ as $n\to \infty$. 
\end{lemma}

\begin{proof}
Note that 
\begin{eqnarray*}
\frac{\rho_{j+1}}{b_{j+1}} &=& \frac{\sigma + 2^{i^*(\rho_j)}  +  \Big( c_j - i^*(\rho_j) - \sigma  \Big) \cdot  \rho_j}{\Big( c_j - i^*(\rho_j) - \sigma \Big) \cdot b_j} = \frac{\Big( c_j - \log_2 \rho_j + O(1) \Big) \cdot  \rho_j}{\Big( c_j - \log_2 \rho_j + O(1) \Big) \cdot b_j} \\
&=& \frac{\rho_j}{b_j} \cdot \left(1+O\left(\frac{1}{c_j - \log_2 \rho_j} \right) \right) = \frac{\rho_j}{b_j} \cdot \left(1+O\left(\frac{1}{\log n} \right) \right),
\end{eqnarray*}
since $c_j - \log_2 \rho_j = \Omega(\log n)$ by Lemma~\ref{lem:i'}. Since $\frac {\rho_1}{b_1} = 2$, we have $\frac{\rho_{j}}{b_{j}} = 2 \cdot \left(1+O\left(\frac{j}{\log n} \right) \right)$. Finally, since $j \le m = o(\log n)$, for all $j \le m$ we have that $\frac{\rho_{j}}{b_{j}} \rightarrow 2$ as $n \to \infty$, and the result follows. 
\end{proof}

Now, we are ready to show that $T_{\rho_m}'$ can  be embedded into a random graph.
 
\begin{lemma}\label{lem:embed} If  $\b < \frac{1}{10 \log 2}$ and $0 < \alpha < \frac{\b \log 2}{2}$, then w.h.p. $\G(n,p)$ contains a copy of $T_{\rho_m}'$. 
\end{lemma}
\begin{proof}
We will embed $T_{\rho_m}'$ in $\G(n,p)$. Select any vertex (arbitrarily) that will serve as a root of the tree. The embedding is done greedily and from the top down, and at each step we reveal the neighbourhood of one vertex. We group vertices in the same level (i.e.\ distance from the root) consecutively. 
The embedding will be determined iteratively as we reveal the random graph. We will not put a vertex of $\G(n,p)$ into our partial embedding until we have exposed all of its children. 

We say that a vertex in level $k$ is \textbf{bad} if its neighbourhood (into the unexposed vertices) is less than $c_{m-k-1}$.  We will show that w.h.p.\ the root is not bad and no vertex has more than $\sigma$ bad children. Any bad children will be put into the partial embedding as leaves, and the other vertices will be arbitrarily assigned (to non-leaves first and then to leaves, if the number of bad children is smaller than $\sigma$). 

The tree $T_{\rho_m} '$ has at most 
$$
\left( \frac{1+\frac{ \eps}{2}}{\log 2} \log n \right)^{m - \alpha \frac{\log n}{\log \log n} } = n^{(1-\alpha) (1+o(1))} = o(n)
$$ 
vertices total in levels $0$ thru $m - \alpha \frac{\log n}{\log \log n} -1$.  Thus, the expected degree (into the unexposed vertices) of each vertex exposed in such a level $k$ is 
\begin{equation}\label{eq:early-avg-deg}
(1+o(1)) d  = (1+o(1)) \frac{1+\eps}{\log 2} \log n > c_k + \frac {\eps}{3 \log 2} \log n
\end{equation}
and so it follows from Chernoff Bound that the probability that a fixed vertex is bad is polynomially small, that is, at most $n^{-\Theta(\eps^2)}$. 
For levels $k$ farther to the bottom, note that the number of vertices that are not embedded yet is always at least $\frac{1+o(1)}{5}n$ and so the expected degree of each exposed vertex in layer $k$ is at least 
$$
\frac{1+o(1)}{5} d = (1+o(1)) \frac{1+\eps}{5 \log 2} \log n > c_k + \frac{1}{10 \log 2} \log n,
$$ 
again yielding that the probability that a fixed vertex is bad is polynomially small (this time the exponent is a universal constant, not a function of $\eps$).

Therefore if $\sigma = \Theta(1/\eps^2)$ is a large enough constant, then w.h.p.\ each vertex has at most $\sigma$ bad children. This proves that our embedding procedure is successful w.h.p.\ and the proof is finished.
\end{proof}

Let $B$ be the set of bereft vertices in $\G(n,p)$ and $R$ be the set of remaining, unexposed vertices that are not embedded into tree yet. Note that $|B| = b_m = \of{\frac{4}{5} + o(1)} n$ and $|R| = n - (1+o(1)) b_m = \of{\frac15 +o(1)}n$. An important property is that no edge between $B$ and $R$ is exposed at this point, so the next Lemma shows that w.h.p.\ set $B$ dominates set $R$ but in such a way that at most one vertex of $T$ is assign to each bereft vertex. This will finish the proof of the main theorem, Theorem~\ref{mainthm}.

\begin{lemma}\label{lem:matching}
W.h.p.\ there is a matching from $R$ to $B$ which saturates $R$.
\end{lemma}
\begin{proof}
We are going to use Hall's theorem for bipartite graphs. It is enough to show that for every subset $S\subseteq R$, Hall's condition holds, that is, we have that $|N(S) \cap B| \ge |S|$.

We will use the following useful upper bound: 
\begin{align}
&\Pr\sqbs{\nexists \textrm{ matching saturating $R$}} \label{eq:probnomatch}\\
&\le\Pr\sqbs{\exists v\in R\,:\, N(v)\cap B = \emptyset}\nonumber \\
&+ \Pr\sqbs{\exists S\subseteq R, T\subseteq B\,:\,|S|=k\ge 2,\, |T|=k-1,\,N(S)\cap B= T,\,e(S:T) \ge 2(k-1)}\nonumber
\end{align}
where $e(S:T)$ represents the number of edges between $S$ and $T$. The first term bounds the probability that Hall's condition fails for some set of size one. To see why the condition in the second term is equivalent to the property that Hall's condition fails for some set of cardinality at least 2 is slightly more complicated.
Take a smallest size $S\subseteq R$ with $|S|\ge 2$, which violates Hall's condition, i.e.\ $|S| > |T|$ where $T=N(S)\cap B$. If $|S|=k$, then $|T| = k-1$, otherwise we could remove some vertex from $S$ to get a smaller set that violates Hall's condition. Every vertex in $T$ must have degree at least $2$ into $S$, because removing a degree $1$ vertex from $T$ and its unique neighbour in $S$ gives us a smaller set which violates Hall's condition. So the number of edges between $S$ and $T$ must be at least $2(k-1).$

In order to bound the first term in~(\ref{eq:probnomatch}), note that 
\begin{eqnarray*}
\Pr\sqbs{\exists v\in R\,:\, N(v)\cap B = \emptyset} &\le& |R|(1-p)^{|B|} \le n\exp\of{-\frac{4}{5\log 2}\log n} \\
&\le& n\exp\of{- 1.15 \log n} = o(1).
\end{eqnarray*}
To bound the second term in~(\ref{eq:probnomatch}), let $Y$ count the number of sets $S$ and $T$ satisfying the condition in this term. Then we may bound the expectation of $Y$ from above by 
\begin{align*}
\E\sqbs{Y}&\le \sum_{k=2}^{|R|}\binom{|R|}{k}\binom{|B|}{k-1}\binom{k(k-1)}{2(k-1)}p^{2(k-1)}(1-p)^{k(|B|-(k-1))} \\
&\le \sum_{k=2}^{|R|}\bfrac{|R|e}{k}^k\bfrac{|B|e}{(k-1)}^{k-1}\bfrac{kep}{2}^{2(k-1)}\exp\of{-pk(|B|-(k-1))},
\end{align*}
since ${a \choose b} \le ( ae/b )^b$. It follows that
\begin{align*}
\E\sqbs{Y}&\le \sum_{k=2}^{|R|}\exp\left((2k-1)\log\bfrac{n}{k} + 2(k-1)\log\bfrac{k\log n}{n} \right. \\
&\qquad\qquad\qquad \left.-\frac{4}{5\log2}k\log n\of{1-\frac{k}{\frac{4}{5}n} + o(1)} +O(k)\right)\\
&\le \sum_{k=2}^{|R|}\exp\of{\log n +2(k-1)\log\log n - 1.15k\log n\of{1-\frac{5k}{4n} + o(1)} + O(k)}.
\end{align*}
For each value of $k$ such that $2 \le k \le |R|$, each term above is $o(n^{-1.1})$. Since we are summing over $|R|=(\frac15 + o(1))n$ many terms, we get $\E\sqbs{Y} = o(1)$ and so $\Pr\sqbs{Y>0} = o(1)$ by Markov's inequality. It follows that the probability in \eqref{eq:probnomatch} is $o(1)$ and the proof is finished.
\end{proof}

\section{Proof of Theorem~\ref{lowerbdthm}}\label{sec:lowerbdthm}

First, let us concentrate on the lower bound. 
In the rest of this section, set $d = p(n-1) = c \log n$ where $0 < c < \frac{1}{\log 2}$. Let $\eps', \eps'' > 0$ be constants such that $c+\eps'+\eps'' < \frac{1}{\log 2}$. Set $c'=c+\eps'$ and $c'' = c'+\eps''$ so that  $c < c'  <c'' < \frac{1}{\log 2}$. Also define the constant  $\gamma:= \ceil{2\of{\frac{4c+2\eps'}{\eps'^2}}}$.

\bigskip

We will need the following property of a random graph $\G(n,p)$.

\begin{lemma}\label{lem:propertyforlower}
The following properties hold w.h.p.
\begin{enumerate}
\item $G(n,p)$ has no vertices of degree at least $4 \log n$.
\item $G(n,p)$ has no paths of length $\gamma$ consisting of vertices of degree at least $c' \log n$.
\end{enumerate}
\end{lemma}

\begin{proof}
(i) follows easily from Chernoff Bound. 

For (ii), we first note that the expected number of paths on $\gamma$ vertices is $O(n^{\gamma} p^{\gamma - 1}) = O(n\cdot (\log n)^{\gamma-1})$. Given such a path, we are looking for each vertex in the path to have at least $c' \log n - 2$ additional neighbours among the $n-\gamma$ other vertices. By Chernoff Bound, the probability that one vertex in the path has enough neighbours is at most $$\exp \of{- \frac{(\frac{\eps'}{c})^2 c \log n}{2+ \frac{\eps'}{c}} \cdot (1-o(1)) } \le \exp \of{- \frac{\eps'^2}{4c+2\eps'} \log n}.$$
Hence the expected number of paths on $\gamma$ vertices consisting of vertices with degree at least $c'\log n$, is at most 
$$
O\of{n\cdot(\log n)^{\gamma-1} \cdot \exp \of{- \frac{\eps'^2}{4c+2\eps'} \log n}^{\gamma}} = O\of{n^{-1} \cdot(\log n)^{\gamma-1} } = o(1).
$$
So by Markov's inequality, w.h.p., there are no such paths.
\end{proof}

\bigskip

Now, we are ready to show the lower bound.

\begin{proof}[Proof of the lower bound in Theorem~\ref{lowerbdthm}.]
Suppose that in any graph $G$ of maximum degree at most $4 \log n$, the vertex $v$ can acquire $n^{c'' \log 2}$ weight. Then, it follows from Observation~\ref{obs:min_degree} that $d(v) \ge \log_2 \of{ n^{c'' \log 2}} = c'' \log n$. Furthermore, by averaging argument, some neighbour $u$ of $v$ must have acquired at least $\frac{n^{c'' \log 2}}{4 \log n}$ weight, and so  
$$
d(u) \ge \log_2 \sqbs{\frac{n^{c'' \log 2}}{4 \log n} } \ge c'' \log n - O(\log \log n).
$$
Applying the same reasoning inductively $\beta = O(1)$ times, we find a path of length $\beta$ of vertices of degree at least $c'' \log n - O(\log \log n) \ge c' \log n$.

But by Lemma~\ref{lem:propertyforlower}, w.h.p.\ $\G(n,p)$ has max degree at most $4 \log n$ and no long path of high degree vertices, and so w.h.p.\ no vertex can ever get a weight more than $n^{c'' \log 2}$. Thus, at least $n^{1-c'' \log 2}$ many vertices have nonzero weight after any legal sequence of moves and the lower bound holds. 
\end{proof}

For the upper bound in Theorem~\ref{lowerbdthm}, we must show that all the weight can be pushed to at most $n^{1-c+\eps}$ many vertices. To do this we basically follow the embedding proof from Theorem~\ref{mainthm} but this time with many roots. We sketch the idea of this process, and the reader may check the details.

\begin{proof}[Proof sketch of the upper bound in Theorem \ref{lowerbdthm}] Suppose $p= \frac{c+o(1)}{\log 2} \cdot \frac{\log n}{n}$ for some $c \in (0,1)$. We would like to show that for any $0 < \eps' < \min\{c, 1-c\}$, w.h.p., $a_t(\G(n,p)) \le n^{1-c+\eps'}$. We let $\eps = \eps' / 2$ and prove that $a_t(\G(n,p)) \le n^{1-c+\eps + o(1)}$.
Let  $\alpha, \beta > 0$ be any two constants such that $\beta < \frac{c}{10\log 2}$, $\alpha < \frac{\beta\log 2}{2}$ and let \[c_j^* := \begin{cases} 
\b \log n  &\textrm{ if } j \le \alpha \frac{\log n}{\log \log n} \\
  \frac{c-\frac{ \eps}{2}}{\log 2} \log n&otherwise.
 \end{cases}\]
Let $\rho$ be the sequence defined as in Defintion~\ref{def:sequence} with respect to this sequence $c^*$. Let $m$ be the largest integer such that $\rho_m \le n^{c-\eps}$. Then $\rho_m = \Omega( n^{c-\eps} /  \log n)$  and $m = (c-\eps+o(1))\frac{\log n}{\log \log n}$.  At this point in the proof of Theorem~\ref{mainthm}, we would adjust the sequence $c^*$ to get a precise value for $\rho_m$. However, in this case, this step is unnecessary since we may simply adjust the number of copies of $T_{\rho_m}'$ which we embed.  Let $L = L(n)$ be an integer such that $L\cdot|T_{\rho_m}'| =L\cdot(1/2+o(1)) \rho_m=(\frac45 + o(1))n$.  Then $L = n^{1-c+\eps+o(1)}.$ We would like to grow $L$ vertex disjoint copies of $T_{\rho_m}'$. To do this we must have $L$ roots. We begin with $2L$ many vertices which are candidate roots. The probability that a fixed vertex has less than $\frac{c-\eps/2}{\log 2}\log n$ neighbours (among the other $n-2L$ vertices) is at most $n^{-\Theta(\eps^2)}$ by Chernoff Bound. So by Markov's inequality,  w.h.p., at least $L$ of these $2L$ vertices have at least $\frac{c-\eps/2}{\log 2}\log n$ neighbours, and we take these $L$ vertices to be the roots. The other $L$ vertices whose neighbourhoods were exposed now play no part in the embedding and can retain their weight until the end. 

We may now proceed as in Lemma~\ref{lem:embed}. We embed the trees from the top down and group vertices in the same level (this time, distance from their respective root) consecutively. In levels $0$ thru $m- \alpha \frac{\log n}{\log\log n} - 1$, the $L$ trees have at most
\[L\cdot \of{\frac{c-\frac{\eps}{2}}{\log 2}\log n}^{m - \alpha\frac{\log n}{\log\log n}} = n^{1-\alpha + o(1)} = o(n)\]
vertices total. So, as in \eqref{eq:early-avg-deg}, the probability that a fixed vertex is bad is polynomially small. We also have that the number of unexposed vertices is always at least $\frac{1+o(1)}{5}n$, so the polynomial bound on the probability that a vertex is bad holds for levels further down as well. So by taking $\sigma$ to be a large enough constant, we successfully embed the $L$ trees w.h.p.

We are now in the situation where we have $\of{\frac45 +o(1)}n$ bereft vertices $B$, and $\of{\frac15 +o(1)}$ unexposed vertices, $R$. Note that since $p= \frac{c+o(1)}{\log 2}\cdot \frac{\log n}{n}$ and $c$ may be small here, we will not be able to guarantee that there is a matching from $R$ to $B$ which saturates all vertices in $R$. However, it is sufficient to find a matching which saturates all but $n^{1-c + \eps}$ vertices in $R$, since these remaining $n^{1-c+\eps}$ vertices can keep their weight. Indeed, if such a matching is found then we have shown that 
\[a_t(\G(n,p)) \le L + L + n^{1-c+\eps} = n^{1-c+\eps + o(1)}.\]
The first $L$ represents the candidate roots which were discarded, the second $L$ represents the roots of the $T_{\rho_m}'$ which were embedded and which receive all the weight from their trees, and the last term represents the unmatched vertices from $R$.

To show that such a matching exists, we may use the defect version of Hall's Theorem: If $|N(S)\cap B| \ge |S| - q$ for all $S \subseteq R$, then there is a matching which saturates all but $q$ vertices of $R$. Emulating the proof of Lemma~\ref{lem:matching} using this version of Hall's Theorem with $q = n^{1-c+\eps}$ proves the existence of the desired matching.
\end{proof}

\section{Proof of Theorem~\ref{thm:almostalltrees}}\label{sec:almostalltrees}

Before we move to the proof of this result, let us mention that our goal is to provide a simple proof of the conjecture and the constant can be easily improved with more effort.

\begin{proof}[Proof of Theorem~\ref{thm:almostalltrees}]
We say that a subgraph $L=\{v-w-x-y\}$ of a tree $T$ is a \textbf{long leaf} if $L$ is an induced path of length~3; in particular, $\deg(v)=1$, $\deg(w) = \deg(x)=2$ in $T$. Observe that the acquisition number of every graph is bounded from below by the number of long leaves. Indeed, it is straightforward to see that, regardless of a strategy used, for every long leaf $L$ we have that at least one vertex from $\{v,w,x\}$ has to have non-zero weight at the end of the process.

Consider the probability space $\Omega$ of all labelled trees of order $n$ uniformly distributed. Let $T$ be a randomly chosen tree from $\Omega$. Clearly $|\Omega| = n^{n-2}$, due to Cayley's formula, so for every fixed tree $T_0$ on $n$ vertices we have $\Pr(T=T_0)=1/n^{n-2}$.  Our goal is to show that a.a.s.\ the number of long leaves in $T$ is at least $n/(3e^3)$.

Let $X_v = X_v(T)$ be an indicator random variable defined as follows:
\[
X_v = 
\begin{cases}
1& \text{if $v$ is a vertex of degree~1 in a long leaf},\\
0& \text{otherwise.}
\end{cases}
\]
Let $X=X(T)$ be a random variable counting the number of long leaves in $T$, that is, $X=\sum_{v \in V(T)} X_v$. Note that for every $v \in V(T)$
\begin{eqnarray*}
E(X_v) &=& \Pr(X_v = 1) = \frac{(n-1)(n-2)(n-3)(n-3)^{n-5}}{n^{n-2}}\\
&=& (1+o(1)) \left( 1 - \frac 3n \right)^n = (1+o(1)) \frac{1}{e^3},
\end{eqnarray*}
since there are $(n-1)(n-2)(n-3)$ choices for the vertices of the long leaf and there are $(n-3)^{n-5}$ ways to embed a tree on remaining vertices. Hence,
\begin{equation}\label{eq:tree0}
E(X) = \sum_{v \in V(T)} E(X_v) = (1+o(1)) \frac{n}{e^3}.
\end{equation}

Now we are going to apply Chebyshev's inequality to show that a.a.s.\ $X\ge \frac{E(X)}{2} \ge \frac{n}{3e^3}$. It follows that
\[
\Pr\left(X\le \frac{E(X)}{2}\right) \le \Pr\left(|X-E(X)| \ge \frac{E(X)}{2}\right) \le \frac{Var(X)}{\frac{1}{4}(E(X))^2}
= 4\left( \frac{E(X^2)}{(E(X))^2} - 1\right).
\]
Hence, it suffices to show that $\frac{E(X^2)}{(E(X))^2}$ tends to~1 as $n\to \infty$. Clearly,
\begin{equation}\label{eq:tree1}
E(X^2) = \sum_{v,v'} E(X_v X_{v'}) = E(X) + \sum_{v\neq v'} E(X_v X_{v'}) = E(X) + \sum_{v\neq v'} \Pr(X_v=X_{v'}=1),
\end{equation}
where the sums are over ordered pairs. Now, for fixed vertices $v\neq v'$,
\begin{equation}\label{eq:tree2}
\Pr(X_v=X_{v'}=1) = \frac{(n-2)(n-3)(n-4)(n-5)(n-6)^2(n-6)^{n-8}}{n^{n-2}} = (1+o(1)) \frac{1}{e^6},
\end{equation}
since there are $(n-2)(n-3)(n-4)(n-5)$ choices for vertices $w,x,w',x'$ in the two corresponding long leaves $L=\{v-w-x-y\}$ and $L'=\{v'-w'-x'-y'\}$, and $(n-6)^2$ choices for $y,y'$ (note that it might happen that $y=y'$ but other than that the two leaves cannot overlap). 
Consequently, \eqref{eq:tree0}, \eqref{eq:tree1}, and \eqref{eq:tree2} imply that
\[
\frac{E(X^2)}{(E(X))^2} = \frac{(1+o(1)) \frac{n}{e^3} + (1+o(1)) n^2 \frac{1}{e^6} }{\left( (1+o(1)) \frac{n}{e^3}\right)^2} = 1+o(1),
\]
as required. The proof of the theorem is finished.
\end{proof}

\section{Concluding Remarks}\label{sec:conclusion}

In this paper, we showed that $p = \frac{1}{\log 2} \cdot \frac{\log n}{n}$ is the sharp threshold for the property $a_t(\G(n,p))=1$. However, precise behaviour of the total acquisition number in the critical window is not determined and it is left as an open problem. We analyzed sparser graphs showing that for $c \in (0,1)$, w.h.p.
$$
\log_n a_t \left( \G \left( n,\frac{c}{\log 2} \cdot \frac{\log n}{n} \right) \right) \sim 1-c,
$$
so the exponent of the total acquisition number is determined up to $o(1)$ term. It also remains to be analyzed and better understood. 

On the other hand, it is not difficult to see when this graph parameter becomes sub-linear. It was already anticipated by
West~\cite{West1, West2} that $a_t(G)$ is linear for $p = c/n$ for any constant $c>0$ and sub-linear for $p \gg 1/n$. This is true, since for $p = c/n$ we have $\Omega(n)$ isolated vertices w.h.p.\ (see, for example,~\cite{JLR}), and so the total acquisition number is linear w.h.p.
For $p = \omega/n$, where $\omega\to\infty$ the domination number is known to be equal to $(1+o(1))n \log \omega / \omega = o(n)$ w.h.p.~\cite{GLS}, so the total acquisition number is also sub-linear w.h.p.

\end{document}